\documentclass[a4paper,fleqn]{cas-dc}
\usepackage[numbers, sort&compress]{natbib}
\def\tsc#1{\csdef{#1}{\textsc{\lowercase{#1}}\xspace}}
\tsc{WGM}
\tsc{QE}

\usepackage{url}
\usepackage{amssymb}
\usepackage{booktabs}       
\usepackage{amsfonts}       
\usepackage{nicefrac}       
\usepackage{microtype}      
\usepackage{lipsum}         
\usepackage{graphicx}
\usepackage{multirow, amsmath, amsthm}

\newtheorem{theorem}{Theorem}[section]
\newtheorem{lemma}[theorem]{Lemma}

\usepackage[colorlinks]{hyperref}

\begin{document}
\let\WriteBookmarks\relax
\def\floatpagepagefraction{1}
\def\textpagefraction{.001}

\shorttitle{SVD-inv}    

\shortauthors{Y. Zhang \textit{et~al.}}  

\title[mode=title]{Differentiable SVD based on Moore-Penrose Pseudoinverse for Inverse Imaging Problems}%

\author[1]{Yinghao Zhang}[orcid=0000-0001-8501-6180]
\cormark[1]
\author[1]{Yue Hu}[orcid=0000-0002-4648-611X]
\cormark[1]
\cortext[cor1]{e-mail: yhao-zhang@stu.hit.edu.cn; huyue@hit.edu.cn}
\address[1]{School of Electronics and Information Engineering, Harbin Institute of Technology, Harbin, China}

\begin{abstract}
	Low-rank regularization-based deep unrolling networks have achieved remarkable success in various inverse imaging problems (IIPs). However, the singular value decomposition (SVD) is non-differentiable when duplicated singular values occur, leading to severe numerical instability during training. In this paper, we propose a differentiable SVD based on the Moore-Penrose pseudoinverse to address this issue. To the best of our knowledge, this is the first work to provide a comprehensive analysis of the differentiability of the trivial SVD. Specifically, we show that the non-differentiability of SVD is essentially due to an underdetermined system of linear equations arising in the derivation process. We utilize the Moore-Penrose pseudoinverse to solve the system, thereby proposing a differentiable SVD. A numerical stability analysis in the context of IIPs is provided. Experimental results in color image compressed sensing and dynamic MRI reconstruction show that our proposed differentiable SVD can effectively address the numerical instability issue while ensuring computational precision. Code is available at \url{https://github.com/yhao-z/SVD-inv}.
\end{abstract}

\begin{keywords}
  differentiable SVD \sep Moore-Penrose pseudoinverse \sep inverse imaging problem
  \end{keywords}

\maketitle



\section{Introduction}
Low-rank regularization has found extensive application in various fields \cite{hu2021low}. Examples include, but are not limited to, recommendation systems \cite{xu2020multi}, visual tracking \cite{sui2018visual}, 3D reconstruction \cite{agudo2018image}, and salient object detection \cite{peng2016salient}. In the context of inverse imaging problems (IIPs), low-rank regularization plays a pivotal role and has demonstrated remarkable effectiveness in applications such as matrix/tensor completion \cite{candes2011robust}, medical image reconstruction \cite{yao2018efficient} and restoration of remote sensing images \cite{zeng2020hyperspectral}. 

Generally, the low-rank regularization method for IIPs consists of two essential components or, in other words, two stages of theoretical construction: 1) the development of a reconstruction model incorporating low-rank regularization, such as the rank minimization model \cite{candes2010power}/Robust Principal Component Analysis (RPCA) model \cite{candes2011robust}. Due to the non-convex nature of rank, the nuclear norm of matrices/tensors \cite{lu2019tensor} is commonly employed as a convex envelope to constrain the rank. 2) the design of iterative optimization algorithms to solve the model, such as the use of Alternating Direction Method of Multipliers (ADMM) \cite{boyd2011distributed}, Proximal Gradient Descent (PGD) \cite{beck2009fast}, and similar algorithms. 

Specifically, the degradation model of IIPs can be commonly described as follows:
\begin{equation}
	\mathsf{A}(X)+n = b,
\end{equation}
where $\mathsf{A}(\cdot)$ is the degradation operator, $X$ is the to-be-reconstructed variable, $b$ is the observed data, and $n$ is the noise. The goal of IIPs is to recover the unknown variable $X$ from the observed data $b$. 
To succinctly illustrate, we take the example of a matrix inverse problem, where \(X\) is a two-dimensional matrix. With the introduction of low-rank regularization, the reconstruction problem can be described as follows:
\begin{equation}
	\min_{X} \frac{1}{2}\|\mathsf{A}(X)-b\|_{2}^{2}+\lambda\|X\|_{*},
\end{equation}
where $\|\cdot\|_{*}$ denotes the nuclear norm of the matrix, and $\lambda$ is the regularization parameter. Adopting the PGD algorithm \cite{cai2010singular,beck2009fast} to solve the above problem, the iterative update formula is as follows:
\begin{equation}
	\label{pgdintro}
	\begin{cases}
		Z = X^{n}-\rho \mathsf{A}^H\left(\mathsf{A}(X^{n})-b\right) \\
		X^{n+1} = \operatorname{SVT}(Z, \lambda\rho). \\
	\end{cases}
\end{equation}
where the subscript indicates the iteration step, $Z$ is the intermediate variable, $\rho$ is the stepsize of the gradient descent step, $\mathsf{A}^H$ is the adjoint of $\mathsf{A}$, SVT denotes singular value thresholding algorithm \cite{cai2010singular}, and $\operatorname{SVT}(Z, \lambda\rho) = U \max(S-\lambda\rho, 0)V^H$ with the singular value decomposition (SVD) $Z=USV^H$.

In the process of solving nuclear norm regularization problems, performing SVD on the matrix and applying a soft-thresholding operation to the decomposed singular values is an efficient and concise approach \cite{cai2010singular}. In iterative optimization algorithms, the computation of SVD is only concerned with its forward process, and it does not involve the backward process requiring SVD derivatives. Consequently, the instability in the derivatives of SVD \cite{wang2021robust} does not adversely affect the normal operation of the algorithm.

Recently, Deep Unrolling Networks (DUNs) \cite{monga2021algorithm,zhang2023physics} have introduced a third stage on top of the traditional two-stage construction of iterative optimization algorithms. Specifically, DUNs unfold the iterative algorithm, such as \eqref{pgdintro}, into a deep neural network. In this process, a fixed number of (typically 10) iterations of the optimization algorithm are implemented within a deep learning framework. This enables the construction of an end-to-end neural network and facilitates adaptive learning of hyperparameters within the iterative algorithm using supervised learning. Furthermore, DUNs can replace certain transformations within the iterative algorithm with neural networks, enhancing the ability to learn from data. Overall, DUNs leverage the physics-driven capabilities of traditional algorithms while harnessing the powerful learning capabilities of deep learning. As a result, they have demonstrated remarkable performance in various IIPs \cite{liang2019deep,liu2019deep}.

Given the widespread application of traditional low-rank regularization models, the natural extension of DUNs to iterative algorithms, and the significant success achieved by DUNs, it is inevitable for low-rank regularization-based DUNs to become a research hotspot \cite{solomon2019deep,zhang2020video,zhang2022t}. However, in DUNs, computing SVD introduces challenges regarding the backpropagation of SVD. The derivative of SVD is undefined when there are repeated singular values \cite{wang2021robust,chen2021efficient}, and it can even result in excessively large derivatives, leading to overflow for a certain training data type (e.g., when using Float32 in PyTorch, the maximum value without overflow is $3.41\times 10^{38}$) when two singular values are very close. This numerical instability makes the training process of DUNs prone to collapse. In some cases, the training process of DUNs can enter a cycle of numerical instability, rendering the training process unable to proceed normally. Therefore, addressing the backpropagation issue of SVD is a crucial step for enabling low-rank regularization-based DUNs.

While the paper \cite{wang2021robust} attempted to address the non-differentiability issue of SVD, it mainly focused on the SVD of positive semi-definite matrices or, in other words, the eigenvalue decomposition of matrices. However, in IIPs, the commonly used SVD is often applied to trivial matrices for SVT. To the best of our knowledge, the differentiability analysis of the trivial SVD remains unexplored. In this paper, we start from the derivation of the derivative of SVD and identify that the non-differentiability issue of SVD is essentially attributed to an underdetermined system of linear equations arising in the derivation process, particularly in the presence of repeated singular values. To overcome this, we propose a differentiable SVD by utilizing the Moore-Penrose pseudoinverse to obtain the minimum norm least-squares solution of the system. Additionally, we provide a Numerical stability analysis for the derivative of this differentiable SVD in the context of IIPs. 

To demonstrate the applicability of our approach, we evaluate our proposed differentiable SVD in two IIPs, i.e., color image compressive sensing and dynamic magnetic resonance imaging (MRI) reconstruction. A state-of-the-art low-rank regularization-based DUN \cite{zhang2020video,zhang2022t} is employed as the baseline model. The experimental results show that our proposed differentiable SVD can effectively address the numerical instability issue of SVD in DUNs, thereby improving the performance.



\section{Related Works}
\subsection{DUNs with Low-Rank Regularization for IIPs}
DUNs with low-rank regularization have obtained substantial success in a wide of applications \cite{monga2021algorithm,zhang2023physics,liang2019deep,liu2019deep}. Similar to traditional iterative optimization methods, they can generally be classified into two categories. The first, here termed LR-DUNs, involves directly applying low-rank regularization to the target of interest, denoted as $X$, for reconstruction. The second category, termed RPCA-DUNs, decomposes $X$ into a low-rank component $L$ and a sparse component $S$. In this case, low-rank regularization is applied to the separated component $L$.

LR-DUNs have demonstrated remarkable effectiveness in various fields such as matrix completion \cite{shanmugam2023unrolling}, high dynamic range imaging \cite{mai2022deep}, hyperspectral image denoising \cite{zhao2023hyperspectral}, spectral CT image reconstruction \cite{chen2023soul}, MRI image reconstruction \cite{ke2021learned,zhang2022t}, and more. These methods mostly employ the matrix nuclear norm regularization for low-rank constraints. However, when dealing with tensor data, unfolding it into a matrix often disrupts its high-dimensional low-rank structure. To address this issue, \citet{zhang2022t,zhang2023dynamic} proposed a framework for LR-DUNs based on tensor singular value decomposition. This approach utilizes the low-rank characteristics of tensors in the CNN-transformed domain and achieves state-of-the-art performance in high-dimensional MRI reconstruction tasks. A similar idea has also been evaluated in the context of video synthesis applications \cite{zhang2020video}. Additionally, RPCA-DUNs have also achieved impressive results in various applications, including but not limited to foreground-background separation \cite{van2021deep}, SAR imaging \cite{an2022lrsr}, radar interference mitigation \cite{ristea2021automotive}, and cloud removal \cite{imran2022deep}. Specifically, methods like CORONA \cite{solomon2019deep} and L+S-Net \cite{huang2021deep} have achieved significant improvement in tasks such as ultrasound clutter suppression and dynamic MRI reconstruction.

The methods mentioned above primarily employ SVT to leverage the low-rank priors. This highlights the simplicity, efficiency, and ease of use characteristics of SVT. However, SVT requires performing SVD on the input data, and the derivative of SVD becomes numerically unstable when encountering repeated singular values \cite{wang2021robust}. This instability can pose challenges during the training of DUNs, especially considering that the training process involves frequent updates to the internal data of the DUNs through backpropagation. Given the large number of training iterations, often in the order of tens of thousands or even hundreds of thousands, the probability of encountering repeated or very close singular values is quite high. Also, there are alternative methods for effective low-rank utilization, such as partial separation \cite{liang2007spatiotemporal,chen2021efficient}, which decomposes $X \in \mathbb{C}^{m \times n}$ into $UV$, where $U\in\mathbb{C}^{m\times r}$ and $V\in\mathbb{C}^{r\times n}$. However, this approach necessitates manually selecting the rank $r$ beforehand, which to some extent restricts its performance. 

\subsection{SVD Gradient}
\label{sec:svd_grad}
\citet{townsend2016differentiating} conducted a detailed analysis of the gradient of SVD. \citet{ionescu2015matrix} provided the derivative for the SVD of symmetric matrices, which have a simplified form compared to the derivative of the trivial SVD. However, these methods did not consider the issue of numerical instability associated with SVD.

In addressing the differentiability issue of SVD for positive semi-definite matrices or, equivalently, the differentiability issue of eigenvalue decomposition for matrices, \citet{wang2019backpropagation} observed that employing the power iteration (PI) method \cite{nakatsukasa2013stable} could gradually yield the eigenvalues and eigenvectors of a matrix through a deflation procedure. Therefore, by substituting the PI method for the exact solution of the SVD for symmetric matrices, they could leverage the gradients of PI in the backward pass to circumvent the numerical instability associated with SVD. However, the imprecise solution in the forward pass could impact the performance of the model. Consequently, they proposed a strategy where accurate SVD is used in the forward pass, while the gradients of PI are utilized in the backward pass. Furthermore, \citet{wang2021robust} discovered the derivative of the SVD of positive semi-definite matrices depends on the matrix $\widetilde{K}$ \cite{ionescu2015matrix} with elements:
\begin{equation}
	\widetilde{K}_{ij}= \begin{cases}
		\frac{1}{\sigma_i-\sigma_j}, & i\neq j \\
		\quad0, & i=j
	\end{cases},
\end{equation}
where $\sigma$ denotes the singular values of the matrix. 
It is evident that when two singular values are equal, the corresponding elements of $\widetilde{K}$ will become infinite, leading to numerical instability. \citet{wang2021robust} proposed a solution based on Taylor expansion to address this issue. When $\sigma_i > \sigma_j$, the $K$th degree Taylor expansion of the element $\widetilde{K}_{ij}$ can be expressed as:
\begin{equation}
  \small
  \begin{aligned}
	\widetilde{K}_{ij}&=\frac{1}{\sigma_i-\sigma_j}=\frac1{\sigma_i}\left(\frac1{1-\sigma_j/\sigma_i}\right)\\
	&=\frac1{\sigma_i}\left(1+\left(\frac{\sigma_j}{\sigma_i}\right)^2+\left(\frac{\sigma_j}{\sigma_i}\right)+...+\left(\frac{\sigma_j}{\sigma_i}\right)^K\right),
  \end{aligned}
\end{equation}
where the second equality holds due to the $K$th degree Taylor expansion of function $f(x) = \frac1{1-x}, \quad x\in[0,1)$ at the point $x=0$. At the element where $\sigma_i < \sigma_j$, we have,
\begin{equation}
	\widetilde{K}_{ij}=\frac{1}{\sigma_i-\sigma_j}=-\frac1{\sigma_j}\left(\frac1{1-\sigma_i/\sigma_j}\right),
\end{equation}
and the same $K$th degree expansion could be used.

However, this approximation method encounters three issues:
1) The Taylor series of the function \(f(x) = \frac{1}{1-x}\) has a convergence radius of \(|x| < 1\). The starting point for applying Taylor expansion is to address the scenario when two singular values are equal, i.e., \(\frac{\sigma_j}{\sigma_i}=1\). In such cases, it cannot be guaranteed that \(\left(1+\left(x\right)^2+\left(x\right)+...+\left(x\right)^N\right), N\rightarrow\infty\) converges to \(\frac{1}{1-x}\), thereby rendering the \(k\)-th order approximation, $\left(1+\left(\frac{\sigma_j}{\sigma_i}\right)^2+\left(\frac{\sigma_j}{\sigma_i}\right)+...+\left(\frac{\sigma_j}{\sigma_i}\right)^K\right)$, ineffective in approximating \(\left(\frac1{1-\sigma_j/\sigma_i}\right)\).
2) During the training process, the repetition of singular values is a probabilistic event and does not occur in every training step. However, this Taylor approximation method alters the gradients of SVD in any scenario, potentially leading to inaccurate backpropagation.
3) Indeed, the $\widetilde{K}$ matrix is not the intrinsic and fundamental reason for the non-differentiability of SVD. We will delve into a detailed analysis in the next section.

To the best of our knowledge, there is currently no comprehensive analysis of the differentiability of the trivial SVD. The work based on Taylor expansion \cite{wang2021robust} can be extended to the trivial SVD. Specifically, the derivatives of the trivial SVD rely on the matrix \(F\) \cite{townsend2016differentiating} with elements:
\begin{equation}
	\label{eq:intro_F}
	F_{ij}= \begin{cases}
		\frac{1}{\sigma_j^2-\sigma_i^2}, & i\neq j \\
		\quad0, & i=j
	\end{cases}.
\end{equation}
Therefore, we can approximate the elements of the \(F\) matrix through a similar \(K\)-th order Taylor expansion. In this paper, we refer to this approach as SVD-taylor. 
SOUL-Net \cite{chen2023soul} employs this SVD-taylor strategy to address the numerical instability associated with the SVT in DUN. This paper is the only one we have found that effectively considers the non-differentiability of SVD in the context of DUN. Moreover, the inspiration for gradient clipping strategies can also be applied to the matrix \(F\). For instance, when two singular values are equal or very close, elements of \(F\) can be clipped to a relatively large value, such as \(1e16\). In this paper, we name this approach SVD-clip. Additionally, through an analysis of the code of TensorFlow \cite{abadi2016tensorflow}, we found that when \(F\) contains infinite values, TensorFlow always sets the corresponding element to zero. We term this approach SVD-tf.
All of the above methods involve adjustments to the \(F\) matrix. However, we observed that the numerical instability of \(F\) is merely a manifestation of the non-differentiability of SVD. The true underlying cause is an underdetermined system of linear equations in the derivation process when facing repeated singular values. Therefore, we employ the Moore-Penrose pseudoinverse to obtain the minimum norm least squares solution to this system, proposing a differentiable SVD, termed SVD-inv.
Moreover, some works \cite{huang2018decorrelated,kessy2018optimal} suggest that smaller-sized matrices have a lower probability of having repeated singular values, making the computation of their SVD gradients more stable. Consequently, they propose dividing the original large-sized matrix into smaller ones for computation. However, this approach involves manual interventions in the original problem and may not accurately measure the global singular vectors \cite{wang2021robust}, leading to suboptimal results.

\section{Differentiable SVD based on Moore-Penrose Pseudoinverse}
In this section, we start from the derivation of the SVD gradient, revealing the deep-seated reason for the non-differentiability of SVD when repeated singular values occur. Specifically, we highlight that a system of linear equations in the derivation process becomes underdetermined in the case of repeated singular values. Subsequently, we utilize the Moore-Penrose pseudoinverse to provide the least squares solution to this system, thereby rendering SVD differentiable.

\subsection{The gradient of the proposed SVD-inv}
Assuming there is an SVD operation for the input $A$ in a deep neural network, i.e., $A = USV^H$, and the network has a loss function \(\mathcal{L}(U,S,V)\) that depends on the SVD decomposition \(U, S, V\). From the backpropagation, We can obtain the partial derivatives of \(\mathcal{L}\) with respect to \(U, S, V\), abbreviated as $\overline{U}, \overline{S}, \overline{V}$. The total derivative of $\mathcal{L}$ can be written as,
\begin{equation}
	\label{eq:big}
	d\mathcal{L} = tr(\overline{U}^HdU)+tr(\overline{S}^HdS)+tr(\overline{V}^HdV) = tr\left(\frac{\partial \mathcal{L}}{\partial A}^HdA\right),
\end{equation}
where $tr(\cdot)$ denotes the trace of the input matrix, and the second equality holds due to that the loss function $\mathcal{L}$ can also be written as $\mathcal{L}(A)$. So, to correctly derive the derivative $\frac{\partial \mathcal{L}}{\partial A}$, it is necessary to express \(dU, dS, dV\) in terms of \(dA\), meaning finding their relationships with \(dA\). Furthermore, by dividing both sides of the equation by \(dA\), the remaining terms will constitute the derivative of \(\frac{\partial \mathcal{L}}{\partial A}\).

\begin{theorem}
	\label{thm:1}
	Given the SVD $A=USV^H$ where $A\in\mathbb{C}^{m\times n}$, $U\in\mathbb{C}^{m\times k}$, $V\in\mathbb{C}^{n\times k}$, $S\in\mathbb{R}^{k\times k}$ and $\operatorname{rank}(A)<k<\min(m,n)$, the following relationships hold,
	\begin{align}
		dU &= U(F\odot[U^HdAVS+SV^HdA^HU] \notag\\
	& \quad +T\odot [U^HdAV]) \label{eq:u} \\
    & \quad +(I_m-UU^H)dAVS^{-1}, \notag \\
		dS &= I_k \odot [U^HdAV], \label{eq:s} \\
		dV &= V\left(F\odot[SU^HdAV+V^HdA^HUS]\right) \notag \\
	& \quad +(I_n-VV^H)dAVS^{-1},\label{eq:v}
	\end{align}
	where $I_k$ represents the \(k \times k\) identity matrix, $\odot$ denotes the Hadamard product,
	$F_{ij}=\begin{cases}
		\quad 0, &\sigma_i=\sigma_j\\
		\frac{1}{\sigma_j^2-\sigma_i^2}, &else
	\end{cases}$
	,
	$T_{ij}=\begin{cases}
		\frac{1}{\sigma_i}, &\sigma_i=\sigma_j\neq 0 \& i\neq j\\
		0, &else
	\end{cases}$,
	and
	$S^{-1}=\begin{cases}
		\frac{1}{\sigma}, \sigma\neq 0\\
		0, else
	\end{cases}$.
\end{theorem}

\begin{proof}
	Taking the total differential of the SVD, we have,
	\begin{equation}
		\label{eq:unitary}
		dA=dUSV^H+UdSV^H+USdV^H.
	\end{equation}
	This equation is the only breakthrough for us to find the relationship between \(dU, dS, dV,\) and \(dA\). However, there are three unknowns here, and only one equation is insufficient to obtain a closed-form solution. We notice that both \(U\) and \(V\) satisfy the following relation,
	\begin{equation}
		U^HU=V^HV=I_k.
	\end{equation}
	%
	Taking the differential of the above equation gives,
	\begin{equation}
		\label{eq:dudvpre}
		dU^HU+U^HdU=dV^HV+V^HdV=0.
	\end{equation}
	Thus, the matrices $d\Omega_U=U^HdU$ and $d\Omega_V=V^HdV$ are skew-symmetric. If we fix an $m\times (m-k)$ matrix $U_\perp$ such that $[U, U_\perp]$ is an unitary matrix, then we may expand $dU$ as,
	\begin{equation}
		\label{eq:du}
		dU=Ud\Omega_U+U_\perp dK_U,
	\end{equation}
	where $dK_U$ is an unconstrained $(m-k)\times k$ matrix. Similarly, we could obtain,
	\begin{equation}
		\label{eq:dv}
		dV=Vd\Omega_V+V_\perp dK_V,
	\end{equation} 
	where $dK_V$ is an unconstrained $(n-k)\times k$ matrix.
	Substituting \eqref{eq:du} and \eqref{eq:dv} into \eqref{eq:dudvpre}, we can easily verify the correctness of both equations. At this point, we have established a well-posed system of three linear equations, i.e., \eqref{eq:unitary}, \eqref{eq:du}, and \eqref{eq:dv}. 
	
	Left-multiplying \eqref{eq:unitary} by $U^H$ and right-multiplying by $V$ gives,
	\begin{equation}
		\label{eq:solve1}
		U^HdAV=d\Omega_US+dS+Sd\Omega_V^H.
	\end{equation}
	Since $d\Omega_U$ and $d\Omega_V$ are skew-symmetric, they have zero diagonal and thus the products $d\Omega_US$ and $Sd\Omega_V^H$ also have zero diagonal. Therefore, we could split \eqref{eq:solve1} into two components: diagonal and off-diagonal. 
	
	Letting $dP:=U^HdAV$ and using $\odot$ as the Hadamard product, the diagonal of \eqref{eq:solve1} is,
	\begin{equation}
		dS=I_k \odot dP=I_k \odot U^HdAV.
	\end{equation}
	Thus, Equation \eqref{eq:s} holds.

	The off-diagonal is,
	\begin{equation}
		\label{eq:solve2}
		\bar{I}_k \odot dP = d\Omega_US-Sd\Omega_V,
	\end{equation}
	where $\bar{I}_k$ denotes the $k\times k$ matrix with zero diagonal and ones else.

	The elemental form of \eqref{eq:solve2} can be written as, 
	\begin{equation}
		\label{eq:solve31}
		(dP)_{ij} = (d\Omega_U)_{ij}\sigma_j-\sigma_i(d\Omega_V)_{ij}, \quad i>j,
	\end{equation}
	and, the corresponding $ji$th element has the following relationship,
	\small{\begin{equation}
		\label{eq:solve32}
		(dP)_{ji}^* = (d\Omega_U)_{ji}^*\sigma_i-\sigma_j(d\Omega_V)_{ji}^*=-(d\Omega_U)_{ij}\sigma_i+\sigma_j(d\Omega_V)_{ij},
	\end{equation}}%
	where the subscript $*$ denotes the conjugate operation.
	From \eqref{eq:solve31} and \eqref{eq:solve32}, we can observe that, due to the skew symmetry of \(d\Omega_U\) and \(d\Omega_V\), it is only necessary to solve for the elements in half of the region of \(d\Omega_U\) and \(d\Omega_V\) (here we choose the lower triangular region). For the element at the \(ij\)-th position, only the \(ij\)-th and \(ji\)-th elements of \(dP\) are relevant. Thus, \eqref{eq:solve31} and \eqref{eq:solve32} form a system of linear equations in two variables, 
	\begin{equation}
		\label{eq:sl}
		\begin{bmatrix}
			\sigma_j &-\sigma_i \\
			-\sigma_i &\sigma_j
		\end{bmatrix}
		\begin{bmatrix}
			(d\Omega_U)_{ij}\\
			(d\Omega_V)_{ij}
		\end{bmatrix}
		=
		\begin{bmatrix}
			(dP)_{ij}\\
			(dP)_{ji}^*
		\end{bmatrix}.
	\end{equation}
	1) When $\sigma_i \neq \sigma_j$, the matrix 
	$\begin{bmatrix}
		\sigma_j &-\sigma_i \\
		-\sigma_i &\sigma_j
	\end{bmatrix}$
	is of full rank, thus we could obtain the closed solution by directly left-multiplying its inverse, i.e,
	\begin{equation}
		\label{eq:neq}
		\begin{aligned}
		\begin{bmatrix}
			(d\Omega_U)_{ij}\\
			(d\Omega_V)_{ij}
		\end{bmatrix}
		&=
		\begin{bmatrix}
			\sigma_j &-\sigma_i \\
			-\sigma_i &\sigma_j
		\end{bmatrix}^{-1}
		\begin{bmatrix}
			(dP)_{ij}\\
			(dP)_{ji}^*
		\end{bmatrix} \\
		&=
		\frac{1}{\sigma_j^2-\sigma_i^2}
		\begin{bmatrix}
			\sigma_j &\sigma_i \\
			\sigma_i &\sigma_j
		\end{bmatrix}
		\begin{bmatrix}
			(dP)_{ij}\\
			(dP)_{ji}^*
		\end{bmatrix}.
	\end{aligned}
	\end{equation}
	2) When $\sigma_i = \sigma_j = 0$, we have $dP_{ij}=dP_{ji}=0$. At these points, \(d\Omega_U\) and \(d\Omega_V\) are independent of \(dP\), or equivalently, \(dA\). Therefore, we can also straightforwardly set them to zero, resulting in a form similar to \eqref{eq:neq}:
	\begin{equation}
		\label{eq:eq=0}
		\begin{bmatrix}
			(d\Omega_U)_{ij}\\
			(d\Omega_V)_{ij}
		\end{bmatrix}
		=
		0
		\begin{bmatrix}
			\sigma_j &\sigma_i \\
			\sigma_i &\sigma_j
		\end{bmatrix}
		\begin{bmatrix}
			(dP)_{ij}\\
			(dP)_{ji}^*
		\end{bmatrix}.
	\end{equation}
	3) When $\sigma_i = \sigma_j = \sigma \neq 0$, the coefficient matrix turns to be
	$\begin{bmatrix}
		\sigma &-\sigma \\
		-\sigma &\sigma
	\end{bmatrix}$, which is a singular matrix, hence non-invertible. Thus, we have revealed that the fundamental reason for the non-differentiability of SVD when repeated singular values occur lies in the fact that the system of linear equations \eqref{eq:sl} does not have an exact solution. In this paper, we propose to use the Moore-Penrose pseudoinverse to obtain the minimum norm least-squares solution, i.e.,%
	\begin{equation}
		\label{eq:eq}
		\begin{bmatrix}
			(d\Omega_U)_{ij}\\
			(d\Omega_V)_{ij}
		\end{bmatrix}
		=
		\begin{bmatrix}
			\sigma &-\sigma \\
			-\sigma &\sigma
		\end{bmatrix}^{\dagger}
		\begin{bmatrix}
			(dP)_{ij}\\
			(dP)_{ji}^*
		\end{bmatrix}
		=
		\begin{bmatrix}
			\frac{1}{2\sigma} &-\frac{1}{2\sigma} \\
			0 &0
		\end{bmatrix}
		\begin{bmatrix}
			(dP)_{ij}\\
			(dP)_{ji}^*
		\end{bmatrix}.
	\end{equation}
	Moreover, substituting $\sigma_i$ and $\sigma_j$ with $\sigma$ in \eqref{eq:solve31} and \eqref{eq:solve32}, we have, $(dP)_{ji}^*=-(dP)_{ij}$, i.e., at this point, \(dP\) also exhibits the antisymmetric property, which makes the two linear equations equivalent, thereby leading to the absence of an exact solution. Substituting this relationship into \eqref{eq:eq}, we ultimately obtain the approximate solution in the case of two equal singular values, that is,
	\begin{align}
		\label{eq:eq2}
		(d\Omega_U)_{ij}&=\frac{(dP)_{ij}}{\sigma},\\
		(d\Omega_V)_{ij}&=0,
	\end{align}

	Therefore, using the equations \eqref{eq:neq}, \eqref{eq:eq=0}, and \eqref{eq:eq2}, we rewrite these relationships in matrix form as follows:
	\begin{align}
		d\Omega_U &= F \odot \left[dPS+SdP^H\right] + T \odot dP, \label{eq:OU} \\
		d\Omega_V &= F \odot \left[SdP+dP^HS\right],\label{eq:OV}
	\end{align}
	where 
	$F_{ij}=\begin{cases}
		\quad 0, &\sigma_i=\sigma_j\\
		\frac{1}{\sigma_j^2-\sigma_i^2}, &else
	\end{cases}$
	, and 
	$T_{ij}=\begin{cases}
		\frac{1}{\sigma_i}, &\sigma_i=\sigma_j\neq 0 \& i\neq j\\
		0, &else
	\end{cases}$.

	Finally, finding $dK_U$ and $dK_V$ in \eqref{eq:du} and \eqref{eq:dv}, we will obtain \eqref{eq:u} and \eqref{eq:v}. Left-multiplying \eqref{eq:unitary} and \eqref{eq:du} by $U_\perp^H$, we could find,
	\begin{equation}
		U_\perp^HdA=dK_USV^H,
	\end{equation}
	and thus,
	\begin{equation}
		U_\perp^HdAV=dK_US.
	\end{equation}
	The diagonal matrix $S$ in our setting may have zero singular values in the last diagonal elements, which means that the corresponding last columns of $dK_US$ are all zeros and thus $U_\perp^HdAV$ is the same. It is indicated that there are no constraints on the last columns of $dK_U$ w.r.t. the zero singular values in $S$, since we have mentioned before in \eqref{eq:du} that $dK_U$ is an unconstrained matrix.  Therefore, we set the unconstrained elements in $dK_U$ as zeros, and then we could obtain,
	\begin{equation}
		\label{eq:KU}
		dK_U=U_\perp^HdAVS^{-1},
	\end{equation}
	where $S^{-1}=\begin{cases}
		\frac{1}{\sigma}, \sigma\neq 0\\
		0, else
	\end{cases}$ and similarly,
	\begin{equation}
		dK_V=V_\perp^HdA^HUS^{-1}.
	\end{equation}

	Substituting \eqref{eq:OU} and \eqref{eq:KU} into \eqref{eq:du} and using the property $U_\perp U_\perp^H=I_m-UU^H$, \eqref{eq:u} holds. Similarly, we can also obtain that \eqref{eq:v} holds.
\end{proof}
\begin{theorem}
	Suppose that at some stage during the computation of the loss function $\mathcal{L}$, we take a matrix $A$ and compute its SVD, i.e., $A=USV^H$. The derivative of $\mathcal{L}$ with respect to $A$ is, 
	\begin{equation}
		\label{eq:lastgradient}
		\begin{aligned}
			\frac{\partial \mathcal{L}}{\partial A}=&U\left[\left(F \odot [U^H\overline{U}-\overline{U}^HU]\right)S+T\odot (U^H\overline{U})\right]V^H \\
			&+(I_m-UU^H)\overline{U}S^{-1}V^H \\
			&+U(I_k \odot \overline{S})V^H\\
      &+ US\left(F \odot [V^H\overline{V}-\overline{V}^HV]\right)V^H \\
	  &+US^{-1}\overline{V}^H(I_n-VV^H)
		\end{aligned}
	\end{equation}
\end{theorem}
\begin{proof}
	Substituting Theorem \ref{thm:1} into \eqref{eq:big} and using the property of the matrix trace, $tr(ABC) = tr(CAB) = tr(BCA)$, we could easily obtain \eqref{eq:lastgradient}. Since the derivation is similar to that in the paper \cite{townsend2016differentiating}, the reader can refer to Section 1.1 in Ref.\cite{townsend2016differentiating} for the detailed proof of this theorem.
\end{proof}

\subsection{Numerical stability analysis in the context of IIPs}
In the context of IIPs where SVT is adopted to utilize the low-rank properties, suppose that at some stage during the computation of the loss function $\mathcal{L}$, we take a matrix $A=USV^H$ and conduct SVT on it to obtain $B=U\hat{S}V^H=U\max(S-\tau,0)V^H$. In general, most of the information in a matrix is retained in the large singular values, while small singular values are thresholded. Therefore, $\tau$ is typically not set too small; here, we choose its value as $\tau > 1e^{-10}$.

\begin{lemma}
	In this specific situation, the partial derivative, $\overline{U}$, have the following form,
	\begin{equation}
		\label{eq:ug}
		\overline{U}=\overline{B}V\hat{S}
	\end{equation}
\end{lemma}
\begin{proof}
	It can be easily obtained using the chain rule and property of the Kronecker product, i.e., $(B^T\otimes A)\operatorname{vec}({X})=\operatorname{vec}({AXB})$.
\end{proof}

Substituting \eqref{eq:ug} into \eqref{eq:lastgradient}, we can obtain,
\begin{equation}
	\label{eq:stableF}
	F \odot [U^H\overline{U}-\overline{U}^HU] = F \odot [U^H\overline{B}V\hat{S}-\hat{S}V^H\overline{B}^HU],
\end{equation}%
\vspace{-0.7cm}
\begin{equation}
	T\odot (U^H\overline{U}) = T\odot (U^H\overline{B}V\hat{S}).\label{eq:stableT}
\end{equation}
Assuming the number of zero elements on the diagonal of the \(k \times k\) matrix \(\hat{S}\) is denoted as \(d\), we refer to the non-zero singular values in \(\hat{S}\) as large singular values, while the singular values that have been thresholded are referred to as small singular values. Therefore, we can discuss the numerical stability of \eqref{eq:stableF} and \eqref{eq:stableT} by dividing the \(k \times k\) size into four parts, as shown in Figure \ref{fig:split}.

\begin{figure}[h]
	\centering
	\includegraphics[scale=0.3]{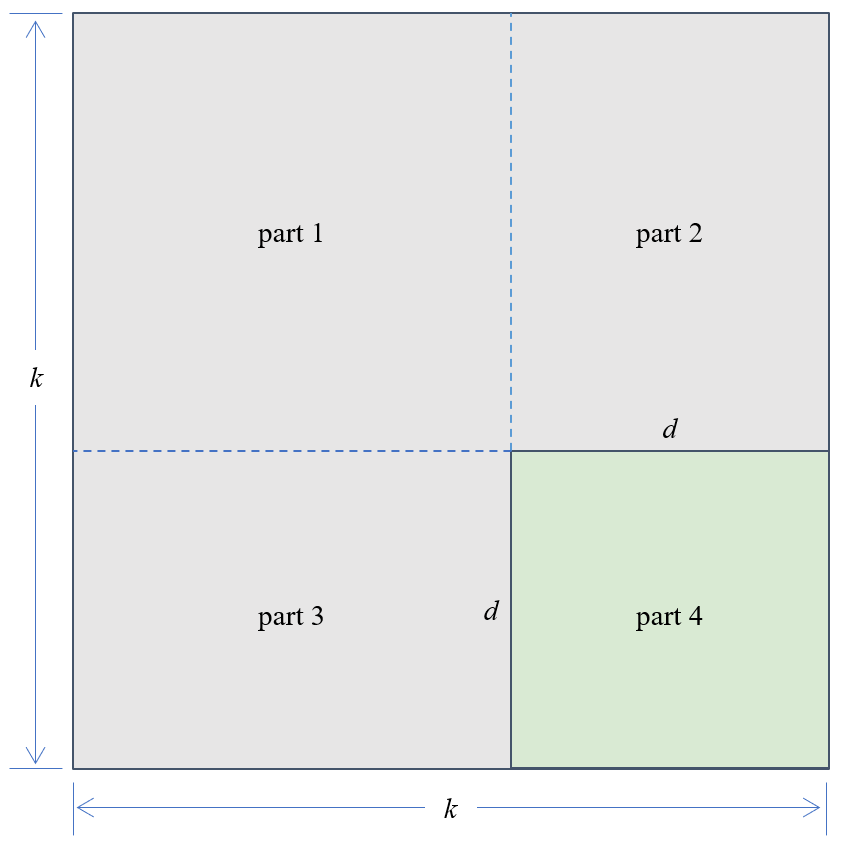}
	\caption[short]{The four split for \(k \times k\) size.}
	\label{fig:split}
\end{figure}

Part 1 corresponds to two large singular values, and the only risk of numerical instability for \(F\) in this scenario is when the two singular values are equal or close. In the actual implementation, a threshold $t$ can be set. For example, set $t$ to \(1e^{30}\) for the Float32 data type and leave some margin with the maximum value of \(3.41 e^{38}\). It should be noted that the value of $t$ can also be directly set to the maximum value. For \(|F_{ij}|>t\), treating the two singular values as equal, the corresponding elements in \(F\) are set to zero, and these elements are transferred to \(T\). Meanwhile, the elements in \(T\) are set to \(\frac{1}{\sigma}\), where \(\sigma\), as a large singular value, should be greater than the threshold \(\tau\) of SVT. Therefore, the elements in \(T\) are less than \(1e^{10}\), ensuring numerical stability.

Parts 2 and 3 correspond to one large and one small singular value. In this case, the two singular values have a significant difference, so there is no involvement of elements in \(T\) as they all remain zeros. The elements in \(F\) are dominated by the large singular value in the denominator, i.e.,
\begin{equation}
	|F_{ij}|=\frac1{\left|\sigma_j^2-\sigma_j^2\right|}=\frac1{\sigma_{large}^2-\sigma_{small}^2} \approx \frac1{\sigma_{large}^2} < 1e^{20},
\end{equation}
ensuring numerical stability.

Part 4 corresponds to two small singular values. Similar to the analysis in Part 1, a threshold $t$ for \(|F_{ij}|\) is set. Elements exceeding this threshold are transferred to \(T\). In this case, due to the small singular values, there is a risk of overflow for the elements \(\frac{1}{\sigma}\) in \(T\). However, it should be noted that the elements of the matrix $[U^H\overline{B}V\hat{S}-\hat{S}V^H\overline{B}^HU]$ in \eqref{eq:stableF} are zeros in part 4, while the elements of the matrix $(U^H\overline{B}V\hat{S})$ in \eqref{eq:stableT} are zeros in both part 2 and 4. This is because the last $d$ elements in the diagonal of $\hat{S}$ are zeros. Therefore, as long as the elements in \(F\) and \(T\) are finite, they will not have any impact on the gradients since they are multiplied by zero. Therefore, we can set all overflow values in \(T\) (in PyTorch, represented as inf) to the maximum value of Float32 to ensure they are finite and thereby maintain numerical stability.

To sum up, we only need to set a sufficiently large threshold \(t\) for \(|F|\), transfer elements exceeding this threshold to \(T\), and set overflow values in \(T\) to the maximum value of Float32. This approach significantly ensures the numerical stability of the SVD-inv derivatives.


\subsection{SVD-inv for positive semi-definite matrices}
The SVD of positive semi-definite matrices finds widespread application in ZCA whitening \cite{kessy2018optimal,huang2018decorrelated}, style transfer \cite{cho2019image}, and second-order pooling \cite{ionescu2015matrix}. In essence, in this particular scenario, SVD serves as a substitute for eigenvalue decomposition. The SVD-taylor method proposed in \cite{wang2021robust} is primarily tailored to this context as well. Hence, we also delve into the application of SVD-inverse on positive semi-definite matrices.

It is known that the SVD of a positive semi-definite matrix \( A \) takes on a special form \( A = USU^H \), where \( U = V \). Consequently, in Equation \eqref{eq:solve31}, we have \( d\Omega_U = d\Omega_V \). When the two singular values $\sigma_i$ and $\sigma_j$ are equal, $dP_{ij}=dP_{ji}=0$ holds. Therefore, as the same as \eqref{eq:eq=0}, $ d\Omega_U$ and $d\Omega_V $ are independent of $dP$ and $F$ could be set to zero at the corresponding points, which is consistent with our analysis. Meanwhile, although the elements in $T$ are not zeros, they are multiplied by zero ($dP_{ij}$) in the final gradient, so they will not affect the gradient. Therefore, the SVD-inv is also applicable to positive semi-definite matrices.

\section{Experiments}
\subsection{Efficacy of SVD-inv}
To demonstrate the effectiveness of our proposed SVD-inv, we designed a forward propagation process embedding SVD, followed by backpropagation to compute the SVD gradient. Specifically, we constructed a matrix \( A \) using Float64 double precision as the input matrix for forward propagation, as shown in Fig.\ref{fig:workflow}.a. This matrix is of size \( 10 \times 10 \). Matrix \( A \) is generated by multiplying three matrices \( U_0 \), \( S_0 \), and \( V_0 \), where both \( U_0 \) and \( V_0 \) are identity matrices, ensuring that the diagonal elements of the diagonal matrix \( S_0 \) are the singular values of \( A \). This setup allows us to assign two singular values of matrix \( A \) that are close to each other by any desired margin. Then, we perform SVD on \( A \) to obtain matrices \( U \), \( S \), and \( V \). It is worth noting that matrix \( S \) is not equal to \( S_0 \) because the diagonal elements of \( S_0 \) are unordered, while the singular values in the diagonal of \( S \) are arranged in descending order. Similarly, \( U \) and \( V \) are not equal to \( U_0 \) and \( V_0 \). Next, three different workflows were examined, as shown in Fig.\ref{fig:workflow}.c-e. In Workflow 1, we do not perform any operations on the singular value matrix \( S \). Workflows 2 and 3 correspond to hard thresholding and soft thresholding of the singular values, respectively, with only the last two singular values set to zero. The processed singular value matrix is denoted as \( \hat{S} \). Finally, we use the \( L_1 \) norm of \( U\hat{S}V^H \) as the loss function.

\begin{figure*}[htbp]
	\centering
	\includegraphics[scale=1]{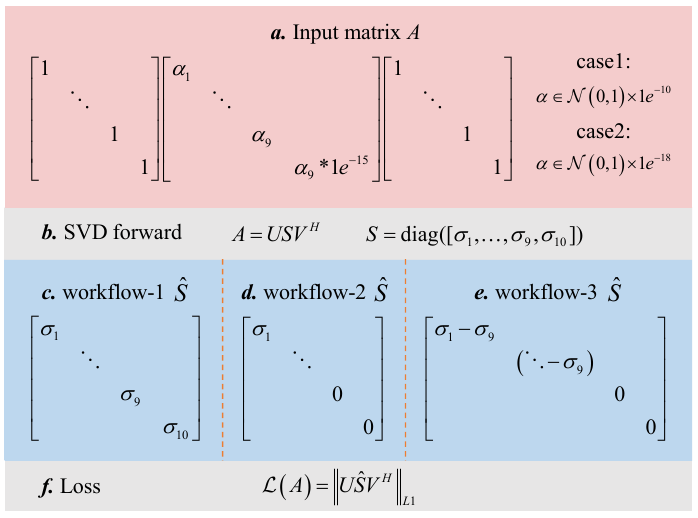}
	\caption{The experimental settings for evaluating the efficacy of SVD-inv.}
	\label{fig:workflow}
\end{figure*}

For setting the singular values of the input matrix \( A \), we generate a random number \( \sigma_0 \), then set \( \sigma_1 = \sigma_0 + \sigma_0 \times 10^{-15} \). These values, along with an additional eight random numbers, constitute the singular values of matrix \( A \). The random numbers are drawn from a standard Gaussian distribution and are scaled by \( 10^{-10} \) in case 1 and by \( 10^{-18} \) in case 2. Next, we examine the matrix \( F \) from Eq.(7), whose maximum value is given by \( 1 / (\sigma_1^2 - \sigma_0^2) \). In case 1, the maximum value is approximately \( 10^{35} \), which is close to the overflow threshold for Float32 precision. Because matrix \( F \) is multiplied by many matrices during gradient computation, there is a high likelihood of overflow when backpropagation is performed in Float32 precision. In case 2, the maximum value is approximately \( 10^{51} \), exceeding the Float32 range, which means overflow will inevitably occur during gradient computation. In other words, when computing the gradient in Float32 precision, case 1 represents a scenario where two singular values are extremely close, while case 2 represents a scenario where the two singular values are effectively equal. It is important to remember that matrix \( A \) was defined with Float64 precision. Therefore, for both case 1 and case 2, we can obtain reference gradient values without the risk of overflow when using Float64 precision.

In Fig.\ref{fig:workflow}.b, we used the Float64 precision SVD and four differentiable SVD methods under Float32 precision, namely SVD-tf, SVD-clip, SVD-taylor, and SVD-inv. Using the automatic differentiation mechanism in PyTorch, we calculated the gradients for these four methods and compared them to the label gradients obtained from the Float64 precision SVD, computing the mean squared error (MSE) of the gradient errors. We generated 1000 random matrices \( A \) and accumulated the MSE gradient errors for the four SVD methods. The results for the two cases and three workflows are shown in the Tab.\ref{tab:efficacy}. From the results, it is evident that our proposed SVD-inv method exhibits the smallest error, which can be attributed to the theoretical guarantees provided by the Moore-Penrose pseudoinverse.

\begin{table}[htbp]
	\centering
	\caption{Cumulative MSE of gradient errors for 1000 experiments for four differentiable SVD methods in Float32 against the standard SVD in Float64.}
	\resizebox{\linewidth}{!}{\begin{tabular}{cccccc}
	  \toprule
	  case  & workflow & SVD-tf & SVD-clip & SVD-taylor & SVD-inv \\
	  \midrule
	  \multirow{3}[2]{*}{1} & 1     & 1412.5 & 1412.5 & 1828.6 & 38.9 \\
			& 2     & 1086.6 & 1086.6 & 6045.1 & 30.2 \\
			& 3     & 1347.7 & 1347.7 & 1663.8 & 640.3 \\
	  \midrule
	  \multirow{3}[2]{*}{2} & 1     & 1611.7 & 1611.7 & 9487.2 & 37.1 \\
			& 2     & 4945.1 & 4945.1 & 15343.9 & 3833.0 \\
			& 3     & 1443.9 & 1443.9 & 7972.0 & 712.0 \\
	  \bottomrule
	  \end{tabular}}%
	\label{tab:efficacy}%
  \end{table}%

\subsection{Applications in IIPs}
We conducted experiments for our SVD-inv in two applications: color image compressive sensing and dynamic MRI reconstruction. We chose these two applications because their target recovered images are tensors, exhibiting strong low-rank characteristics. In contrast, if the images were two-dimensional, their inherent low-rank nature might not be as evident. We utilized the tensor nuclear norm under tensor singular value decomposition to enforce low-rank constraints, constructing the reconstruction models. Subsequently, we employed the ADMM algorithm for solving, ultimately unfolding into a DUN \cite{zhang2020video,zhang2022t}. Within DUN, the utilization of low-rank constraints is achieved through SVT, where SVD is indispensable.

\subsubsection{Color image compressive sensing}
The reconstructed model can be formulated as follows:
\begin{equation}
	\label{eq:cs}
	\min_{\mathcal{X}} \left\|P_\Omega(\mathcal{X})-\mathcal{Y}\right\|_F^2+\lambda\sum_{i=1}^{3}\left\|\mathsf{T}\left(\mathcal{X}\right)^{(i)}\right\|_*,
\end{equation}
where \( \mathcal{X}\in\mathbb{R}^{H \times W \times 3} \) is the target color image tensor, \( P_\Omega(\cdot) \) is the compressive sensing operator, $\Omega$ denotes the sampling set, \( \mathcal{Y} \) is the measurement image, \( \lambda \) is the trade-off parameter, $\mathsf{T}: \mathbb{R}^{H \times W \times 3} \rightarrow \mathbb{R}^{H \times W \times 3}$ denotes the transformation that features tensor low-rank properties \cite{zhang2022t}, and $\mathsf{T}\left(\mathcal{X}\right)^{(i)} = \mathsf{T}\left(\mathcal{X}\right)(:,:,i) \in \mathbb{R}^{H\times W}$. The first term is the data fidelity term, and the second term is the tensor nuclear norm \cite{zhang2022t}. 

We utilized the ADMM algorithm to solve this problem,
\begin{equation}
	\label{eq:admm}
	\begin{cases}
        &\mathcal{Z}_{n} = \mathsf{T}^H \circ \operatorname{SVT}_{{\lambda}/{\mu}} \circ \mathsf{T}(\mathcal{X}_{n-1} + \mathcal{L}_{n-1}) \\
        &\mathcal{X}_{n} = P_\Omega(\mathcal{Y})+P_{\Omega^C}(\mathcal{Z}_{n}-\mathcal{L}_{n-1})\\
        &\mathcal{L}_{n} = \mathcal{L}_{n-1} - \eta (\mathcal{Z}_{n}-\mathcal{X}_{n})
  \end{cases},
\end{equation}
where $\Omega^C$ denotes the complementary set of $\Omega$, $\mathcal{Z}$ is the auxiliary variable, $\mathcal{L}$ is the Lagrange multiplier, $\mu$ and $\eta$ are the auxiliary parameters. The $\mathcal{Z}_{n}$ step is solved by the transformed tensor singular value thresholding algorithm \cite{zhang2022t}, where $\mathsf{T}^H$ is the adjoint operator of $\mathsf{T}$, $\operatorname{SVT}_{{\lambda}/{\mu}}$ is the SVT operator with the threshold $\lambda/\mu$ for each channel of the input tensor data, and the symbol `$\circ$' denotes the composition of operators or functions.

Then we unfold the ADMM algorithm \eqref{eq:admm} into a DUN \cite{zhang2020video,zhang2022t}, as shown in Figure \ref{fig:net}. The three circular nodes represent the three steps in \eqref{eq:admm}, respectively. As for the $\mathcal{Z}_n$ step, we use two separate CNNs to learn the transformation $\mathsf{T}$ and its adjoint operator $\mathsf{T}^H$. We treat the R, G, and B of the color images as three channels and input them into the CNNs. The two CNNs in $\mathcal{Z}$ layer have the same structure, which is a 3-layer 2D convolutional neural network with 16, 16, and 3 filters, respectively. The kernel size is set to 3, and the stride is set to 1. The SVT operator, the $\mathcal{X}_n$ and $\mathcal{L}_n$ steps are performed entirely according to mathematical matrix operations within the framework of deep learning, except that the parameters $\lambda$, $\mu$, and $\eta$ are set to be learnable parameters in DUN. The input of this DUN is the compressed sensing image, and the output is the reconstructed image. We utilized the mean squared error (MSE) as the loss function for training. 

We adopted all the images of cifar100 dataset \cite{krizhevsky2009learning} to train the network, resulting in a total of 60k training images. We chose 14 benchmark color images \footnote{https://sipi.usc.edu/database/database.php} as the test set, as shown in Figure \ref{fig:cimgs}. The network was trained using the Adam optimizer for 50 epochs with a batch size of 128. The initial learning rate was set to 0.001, and an exponential decay with 0.95 rate was adopted. We implemented the DUN using PyTorch and trained it on a single NVIDIA Geforce 1080Ti GPU.

\begin{figure}
	\centering
	\includegraphics[width=1.0\columnwidth]{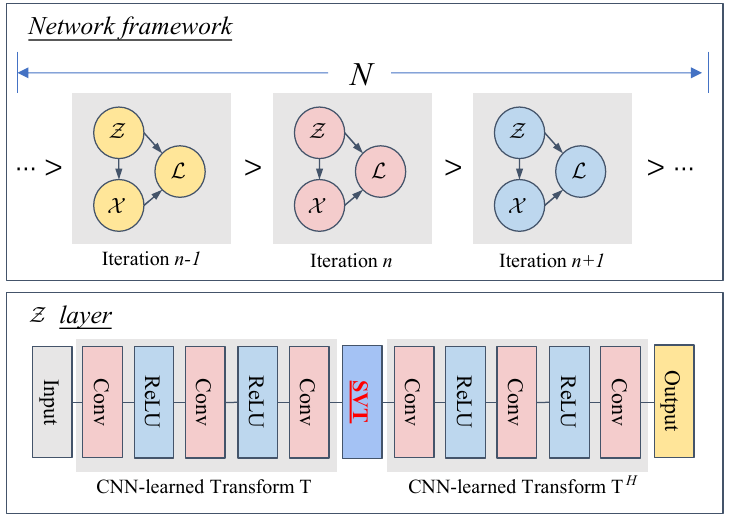}
	\caption[short]{The framework of the used DUN.}
	\label{fig:net}
\end{figure}

\begin{figure}
	\centering
	\includegraphics[width=1.0\columnwidth]{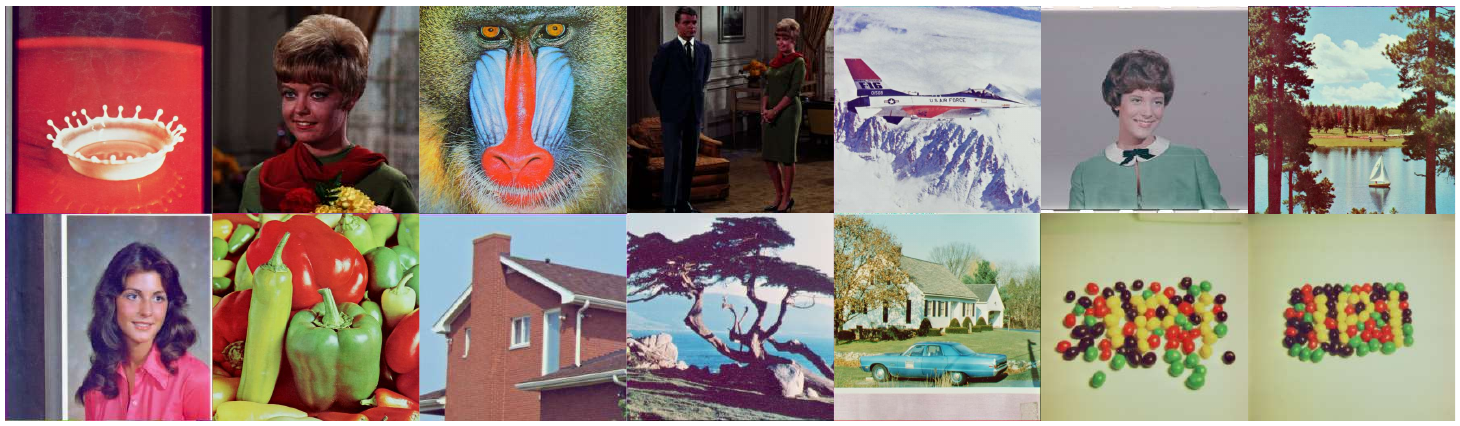}
	\caption[short]{The benchmark color images as the test set.}
	\label{fig:cimgs}
\end{figure}

\subsubsection{Dynamic MRI reconstruction}
The dynamic MRI reconstruction model can be formulated as follows:
\begin{equation}
	\label{eq:mri}
	\min_{\mathcal{X}} \left\|\mathsf{A}(\mathcal{X})-\mathbf{b}\right\|_2^2+\lambda\sum_{i=1}^{T}\left\|\mathsf{T}\left(\mathcal{X}\right)^{(i)}\right\|_*,
\end{equation}
where $\mathsf{A}:\mathbb{C}^{H \times W \times T} \rightarrow \mathbb{C}^{M}$ is the Fourier sampling operator, $T$ denotes the number of frames, $\mathbf{b} \in \mathbb{C}^{M}$ is the measurement data, and the other symbols are the same as in the compressive sensing model. In this experiment, we focus on the single coil scenario, where $\mathsf{A} = \mathsf{S} \circ \mathsf{F}$. $\mathsf{S}$ is the sampling operator like $P_\Omega$ in \eqref{eq:admm}, and $\mathsf{F}$ is the Fourier transform. 

We could use ADMM to solve the above problem, which exhibits a similar form as \eqref{eq:admm} except for the $\mathcal{X}_n$ step. The $\mathcal{X}_n$ step in this case is formulated as follows:
\begin{equation}
	\label{eq:admm2}
	\mathcal{X}_{n} = (\mathsf{A}^H \circ \mathsf{A} + \mu)^{-1}(\mathsf{A}^H(\mathbf{b})+\mu \mathcal{Z}_{n} - \mu \mathcal{L}_{n-1})
\end{equation}

Then, unfolding the ADMM algorithm \eqref{eq:admm2} into a DUN \cite{zhang2022t} results in a similar network structure as shown in Figure \ref{fig:net}. We separate the complex data of MRI images into real and imaginary parts, and treat them as two channels to input into the CNNs. Thus, the CNN in $\mathcal{Z}$ layer is a 3-layer 3D convolutional neural network with 16, 16, and 2 filters, respectively. The kernel size is set to 3, and the stride is set to 1. The output of the first CNN combines the 2 channels into complex data again to input into the SVT operator to exploit the low-rank properties. The input of the DUN is the measurement data, and the output is the reconstructed MRI image. We utilized MSE as the loss function for training.

We utilized the OCMR dataset \cite{chen2020ocmr} to train the network. The dataset contains 204 fully sampled dynamic MRI raw data. We allocated 124 data for training, 40 data for validation, and 40 data for testing. We cropped the images to $128 \times 128 \times 16$ for data augmentation and obtained 1848 fully sampled training samples. The network was trained using the Adam optimizer for 50 epochs with a batch size of 1. Other training settings were the same as in the compressive sensing experiment.
Three types of sampling masks \cite{zhang2022t} are adopted, i.e., pseudo-radial sampling, variable density sampling (VDS) and VISTA. The sampling masks and the example images of OCMR are shown in Figure \ref{fig:ocmr}.

\begin{figure*}
	\centering
	\includegraphics[scale=0.7]{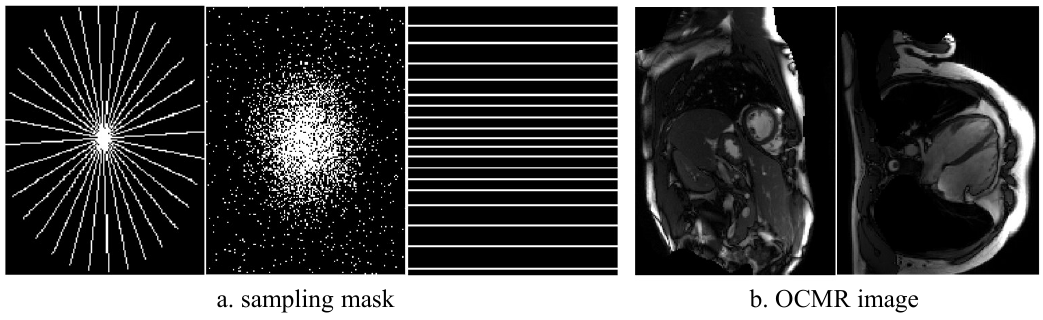}
	\caption[short]{The sampling masks and the example images of OCMR. The sampling masks from left to right are the pseudo-radial sampling, variable density sampling and VISTA mask, respectively.}
	\label{fig:ocmr}
\end{figure*}

\subsubsection{Results}

\begin{table*}[htbp]
	\centering
	\caption{The results of the color image compressive sensing. A range of CS ratios $\{4\%, 10\%, 30\%, 40\%, 50\%\}$ are evaluated. The results are reported by the mean value of PSNR, SSIM, and MSE for all the test images. The MSEs are in the unit of 1e10. The best results are highlighted in bold.}
	\resizebox{\linewidth}{!}{\begin{tabular}{l|ccc|ccc|ccc|ccc|ccc}
	  \toprule
			& \multicolumn{3}{c|}{4\%} & \multicolumn{3}{c|}{10\%} & \multicolumn{3}{c|}{30\%} & \multicolumn{3}{c|}{40\%} & \multicolumn{3}{c}{50\%} \\
			& PSNR↑ & SSIM↑ & MSE↓  & PSNR↑ & SSIM↑ & MSE↓  & PSNR↑ & SSIM↑ & MSE↓  & PSNR↑ & SSIM↑ & MSE↓  & PSNR↑ & SSIM↑ & MSE↓ \\
	  \midrule
	  SVD-torch & \multicolumn{3}{c|}{\textbackslash{}} & \multicolumn{3}{c|}{\textbackslash{}} & \multicolumn{3}{c|}{\textbackslash{}} & \multicolumn{3}{c|}{\textbackslash{}} & \multicolumn{3}{c}{\textbackslash{}} \\
	  SVD-tf & \textbf{22.080} & \textbf{0.284} & \textbf{42.16} & 25.377 & 0.429 & 20.41 & 30.952 & 0.671 & 6.287 & 32.794 & 0.741 & 4.314 & \textbf{34.768} & \textbf{0.803} & \textbf{2.943} \\
	  SVD-clip & 22.080 & 0.284 & 42.16 & 25.377 & 0.429 & 20.41 & 30.952 & 0.671 & 6.287 & 32.794 & 0.741 & 4.314 & 34.768 & 0.803 & 2.943 \\
	  SVD-taylor & 21.877 & 0.285 & 44.45 & \textbf{25.503} & \textbf{0.437} & \textbf{20.10} & 30.837 & 0.666 & 6.439 & 32.707 & 0.740 & 4.388 & 34.583 & 0.800 & 3.036 \\
	  SVD-inv & 22.072 & 0.284 & 42.29 & 25.426 & 0.430 & 20.21 & \textbf{30.952} & \textbf{0.671} & \textbf{6.286} & \textbf{32.795} & \textbf{0.741} & \textbf{4.314} & 34.768 & 0.803 & 2.944 \\
	  \bottomrule
	  \end{tabular}}%
	\label{tab:cs}%
  \end{table*}%

\begin{table*}[htbp]
	\centering
	\caption{The results of the dynamic MRI reconstruction. The results are reported by the mean value of PSNR, SSIM, and MSE for all the test images. The MSEs are in the unit of 1e-5. The best results are highlighted in bold.}
	\resizebox{\linewidth}{!}{\begin{tabular}{l|ccc|ccc|ccc|ccc|ccc}
	  \toprule
			& \multicolumn{3}{c|}{radial-16lines} & \multicolumn{3}{c|}{radial-30lines} & \multicolumn{3}{c|}{vds-8X} & \multicolumn{3}{c|}{vista-8X} & \multicolumn{3}{c}{vista-12X} \\
			& PSNR↑ & SSIM↑ & MSE↓  & PSNR↑ & SSIM↑ & MSE↓  & PSNR↑ & SSIM↑ & MSE↓  & PSNR↑ & SSIM↑ & MSE↓  & PSNR↑ & SSIM↑ & MSE↓ \\
	  \midrule
	  SVD-torch & 40.653 & 0.979 & 10.17 & 42.779 & 0.987 & 6.178 & \multicolumn{3}{c|}{\textbackslash{}} & \multicolumn{3}{c|}{\textbackslash{}} & \multicolumn{3}{c}{\textbackslash{}} \\
	  SVD-tf & \textbf{40.572} & \textbf{0.979} & \textbf{10.33} & 42.873 & 0.987 & 6.046 & 40.557 & 0.979 & 10.48 & 37.381 & 0.960 & 21.78 & \textbf{26.952} & \textbf{0.816} & \textbf{246.1} \\
	  SVD-clip & 40.572 & 0.979 & 10.33 & 42.873 & 0.987 & 6.046 & 40.557 & 0.979 & 10.48 & 37.381 & 0.960 & 21.78 & 26.952 & 0.816 & 246.1 \\
	  SVD-taylor & 40.497 & 0.979 & 10.51 & 42.826 & 0.987 & 6.094 & 40.450 & 0.979 & 10.71 & \textbf{37.846} & \textbf{0.962} & \textbf{19.31} & 26.735 & 0.810 & 252.6 \\
	  SVD-inv & 40.569 & 0.979 & 10.33 & \textbf{42.886} & \textbf{0.987} & \textbf{6.033} & \textbf{40.597} & \textbf{0.979} & \textbf{10.39} & 37.142 & 0.959 & 23.35 & 26.833 & 0.813 & 251.9 \\
	  \bottomrule
	  \end{tabular}}%
	\label{tab:mri}%
  \end{table*}%

We evaluated our proposed SVD-inv with four types of SVD: SVD-torch, SVD-taylor, SVD-tf and SVD-clip. We utilized the built-in SVD function provided by PyTorch to realize both forward and backward propagation for SVD-torch, while the others were implemented by PyTorch in forward propagation and modified the gradients in backward propagation. More details of the compared SVDs have been explained in the introduction. It should be noted that for a fair comparison, we set all random seeds as 3407 \cite{picard2021torch} in the training of all DUNs.

In the application of color image compressed sensing, we evaluated the performance of the DUNs that utilized different SVDs for a range of CS ratios $\{4\%, 10\%, 30\%, 40\%, 50\%\}$. The PSNR and SSIM values of the reconstructed images are shown in Table \ref{tab:cs}. In the application of dynamic MRI reconstruction, we utilized three different sampling masks as mentioned in Figure \ref{fig:ocmr}. Also, different number of lines for radial sampling, i.e., 16 and 30, were adopted. For the VDS and VISTA masks, different acceleration factors are adopted, i.e., 8X and 12X. The quantitative results for all compared SVDs are shown in Table \ref{tab:mri}.

From the results, we can observe that the SVD-tf and SVD-clip exhibit exactly the same performance in both applications and all the detailed cases. In the derivation of these two SVDs, the difference between them and our SVD-inv lies in Eq.\eqref{eq:OU}, i.e., the following relationship holds in their cases:
\begin{equation}
	\label{eq:ou}
	d\Omega_U = F \circ \left[dPS+SdP^H\right],
\end{equation}
where $F$ is in the form of \eqref{eq:intro_F}. In the case of $\sigma_i = \sigma_j = \sigma$, SVD-tf sets the elements in $F$ to zero, while SVD-clip sets them to a relatively large value, such as $1e6$. From \eqref{eq:solve31} and \eqref{eq:solve32}, we can see that if $\sigma_i = \sigma_j$, $(dP)_{ij}=-(dP)^*_{ji}$. Therefore, in the points that two singular values equal, the elements in the matrix $\left[dPS+SdP^H\right]$ are zeros, and the elements in $F$ will not affect the gradients. We can conclude that the SVD-tf and SVD-clip are equivalent. However, it is also worth noting that both SVD-tf and SVD-clip are heuristic, lacking theoretical guarantees when duplicate singular values occur.

Our proposed SVD-inv has achieved better reconstruction results in most cases, which is attributed to the precise calculation of the SVD derivatives. It is worth mentioning that under the radial sampling cases in dynamic MRI reconstruction, SVD-torch is exceptionally able to train all epochs stably. In these cases, we found that SVD-inv can achieve performance nearly identical to that of SVD-torch, which also corroborates the correctness of our derivation. However, SVD-inv and SVD-torch still obtained different quantitative metrics. This may be due to slight differences in specific operations adopted during the actual calculation, such as removing or retaining zero singular values, which could lead to minor differences in the resulting gradients \cite{townsend2016differentiating}. This is completely acceptable. Additionally, among all the quantitative metrics, we noticed that SVD-taylor has a significant difference compared to other SVDs, which is because even when the gradient has high stability (two singular values have significant differences), SVD-taylor still approximates the gradient, inevitably introducing bias. Moreover, its approximation method is only effective under specific conditions, which have been discussed in detail in section \ref{sec:svd_grad}.

\section{Discussion and Conclusion} 
In this paper, we proposed a novel differentiable SVD, which we refer to as SVD-inv. We have revealed that the nondifferentiablity of SVD is due to an underdetermined system of linear equations in the derivation of the SVD gradient. We introduced the Moore-Penrose pseudoinverse to solve this system and obtained the minimum norm least-squares solution, thus proposing the SVD-inv. We have demonstrated that the SVD-inv is applicable to both general matrices and positive semi-definite matrices. A comprehensive analysis of the numerical stability of the SVD-inv in the context of IIPs has been provided. As the unstable gradients of SVD mainly occur when two singular values are equal, our SVD-inv can effectively handle this situation even beyond the context of IIPs, thus having broad applicability. Experiments in two applications, color image compressive sensing and dynamic MRI reconstruction, have verified the effectiveness of the SVD-inv.

\section*{Acknowledgments}
This work is supported by the National Natural Science Foundation of China [grant number 62371167]; the Natural Science Foundation of Heilongjiang [grant number YQ2021F005].

\bibliographystyle{model2-names.bst}
\bibliography{refs}  

\begin{thebibliography}{45}
\expandafter\ifx\csname natexlab\endcsname\relax\def\natexlab#1{#1}\fi
\providecommand{\url}[1]{\texttt{#1}}
\providecommand{\href}[2]{#2}
\providecommand{\path}[1]{#1}
\providecommand{\DOIprefix}{doi:}
\providecommand{\ArXivprefix}{arXiv:}
\providecommand{\URLprefix}{URL: }
\providecommand{\Pubmedprefix}{pmid:}
\providecommand{\doi}[1]{\href{http://dx.doi.org/#1}{\path{#1}}}
\providecommand{\Pubmed}[1]{\href{pmid:#1}{\path{#1}}}
\providecommand{\bibinfo}[2]{#2}
\ifx\xfnm\relax \def\xfnm[#1]{\unskip,\space#1}\fi
\bibitem[{Abadi et~al.(2016)Abadi, Barham, Chen, Chen, Davis, Dean, Devin,
  Ghemawat, Irving, Isard et~al.}]{abadi2016tensorflow}
\bibinfo{author}{Abadi, M.}, \bibinfo{author}{Barham, P.},
  \bibinfo{author}{Chen, J.}, \bibinfo{author}{Chen, Z.},
  \bibinfo{author}{Davis, A.}, \bibinfo{author}{Dean, J.},
  \bibinfo{author}{Devin, M.}, \bibinfo{author}{Ghemawat, S.},
  \bibinfo{author}{Irving, G.}, \bibinfo{author}{Isard, M.}, et~al.,
  \bibinfo{year}{2016}.
\newblock \bibinfo{title}{{TensorFlow}: a system for {Large-Scale} machine
  learning}, in: \bibinfo{booktitle}{12th USENIX symposium on operating systems
  design and implementation (OSDI 16)}, pp. \bibinfo{pages}{265--283}.
\bibitem[{Agudo et~al.(2018)Agudo, Pijoan and Moreno-Noguer}]{agudo2018image}
\bibinfo{author}{Agudo, A.}, \bibinfo{author}{Pijoan, M.},
  \bibinfo{author}{Moreno-Noguer, F.}, \bibinfo{year}{2018}.
\newblock \bibinfo{title}{Image collection pop-up: 3d reconstruction and
  clustering of rigid and non-rigid categories}, in:
  \bibinfo{booktitle}{Proceedings of the IEEE conference on computer vision and
  pattern recognition}, pp. \bibinfo{pages}{2607--2615}.
\bibitem[{An et~al.(2022)An, Jiang, Wu, Teh, Sun, Li and Yang}]{an2022lrsr}
\bibinfo{author}{An, H.}, \bibinfo{author}{Jiang, R.}, \bibinfo{author}{Wu,
  J.}, \bibinfo{author}{Teh, K.C.}, \bibinfo{author}{Sun, Z.},
  \bibinfo{author}{Li, Z.}, \bibinfo{author}{Yang, J.}, \bibinfo{year}{2022}.
\newblock \bibinfo{title}{{LRSR-ADMM-Net}: A joint low-rank and sparse recovery
  network for {SAR} imaging}.
\newblock \bibinfo{journal}{IEEE Transactions on Geoscience and Remote Sensing}
  \bibinfo{volume}{60}, \bibinfo{pages}{1--14}.
\bibitem[{Beck and Teboulle(2009)}]{beck2009fast}
\bibinfo{author}{Beck, A.}, \bibinfo{author}{Teboulle, M.},
  \bibinfo{year}{2009}.
\newblock \bibinfo{title}{A fast iterative shrinkage-thresholding algorithm for
  linear inverse problems}.
\newblock \bibinfo{journal}{SIAM journal on imaging sciences}
  \bibinfo{volume}{2}, \bibinfo{pages}{183--202}.
\bibitem[{Boyd et~al.(2011)Boyd, Parikh, Chu, Peleato, Eckstein
  et~al.}]{boyd2011distributed}
\bibinfo{author}{Boyd, S.}, \bibinfo{author}{Parikh, N.}, \bibinfo{author}{Chu,
  E.}, \bibinfo{author}{Peleato, B.}, \bibinfo{author}{Eckstein, J.}, et~al.,
  \bibinfo{year}{2011}.
\newblock \bibinfo{title}{Distributed optimization and statistical learning via
  the alternating direction method of multipliers}.
\newblock \bibinfo{journal}{Foundations and Trends{\textregistered} in Machine
  learning} \bibinfo{volume}{3}, \bibinfo{pages}{1--122}.
\bibitem[{Cai et~al.(2010)Cai, Cand{\`e}s and Shen}]{cai2010singular}
\bibinfo{author}{Cai, J.F.}, \bibinfo{author}{Cand{\`e}s, E.J.},
  \bibinfo{author}{Shen, Z.}, \bibinfo{year}{2010}.
\newblock \bibinfo{title}{A singular value thresholding algorithm for matrix
  completion}.
\newblock \bibinfo{journal}{SIAM Journal on optimization} \bibinfo{volume}{20},
  \bibinfo{pages}{1956--1982}.
\bibitem[{Cand{\`e}s et~al.(2011)Cand{\`e}s, Li, Ma and
  Wright}]{candes2011robust}
\bibinfo{author}{Cand{\`e}s, E.J.}, \bibinfo{author}{Li, X.},
  \bibinfo{author}{Ma, Y.}, \bibinfo{author}{Wright, J.}, \bibinfo{year}{2011}.
\newblock \bibinfo{title}{Robust principal component analysis?}
\newblock \bibinfo{journal}{Journal of the ACM (JACM)} \bibinfo{volume}{58},
  \bibinfo{pages}{1--37}.
\bibitem[{Cand{\`e}s and Tao(2010)}]{candes2010power}
\bibinfo{author}{Cand{\`e}s, E.J.}, \bibinfo{author}{Tao, T.},
  \bibinfo{year}{2010}.
\newblock \bibinfo{title}{The power of convex relaxation: Near-optimal matrix
  completion}.
\newblock \bibinfo{journal}{IEEE Transactions on Information Theory}
  \bibinfo{volume}{56}, \bibinfo{pages}{2053--2080}.
\bibitem[{Chen et~al.(2020)Chen, Liu, Schniter, Tong, Zareba, Simonetti, Potter
  and Ahmad}]{chen2020ocmr}
\bibinfo{author}{Chen, C.}, \bibinfo{author}{Liu, Y.},
  \bibinfo{author}{Schniter, P.}, \bibinfo{author}{Tong, M.},
  \bibinfo{author}{Zareba, K.}, \bibinfo{author}{Simonetti, O.},
  \bibinfo{author}{Potter, L.}, \bibinfo{author}{Ahmad, R.},
  \bibinfo{year}{2020}.
\newblock \bibinfo{title}{{OCMR} (v1.0)--open-access multi-coil k-space dataset
  for cardiovascular magnetic resonance imaging}.
\newblock \bibinfo{journal}{arXiv preprint arXiv:2008.03410} .
\bibitem[{Chen et~al.(2023)Chen, Xia, Yang, Chen, Liu, Zhou, Wang, Chen, Wen
  and Zhang}]{chen2023soul}
\bibinfo{author}{Chen, X.}, \bibinfo{author}{Xia, W.}, \bibinfo{author}{Yang,
  Z.}, \bibinfo{author}{Chen, H.}, \bibinfo{author}{Liu, Y.},
  \bibinfo{author}{Zhou, J.}, \bibinfo{author}{Wang, Z.},
  \bibinfo{author}{Chen, Y.}, \bibinfo{author}{Wen, B.},
  \bibinfo{author}{Zhang, Y.}, \bibinfo{year}{2023}.
\newblock \bibinfo{title}{{SOUL-Net}: A sparse and low-rank unrolling network
  for spectral {CT} image reconstruction}.
\newblock \bibinfo{journal}{IEEE Transactions on Neural Networks and Learning
  Systems} .
\bibitem[{Chen et~al.(2021)Chen, Yao, Xiao and Wang}]{chen2021efficient}
\bibinfo{author}{Chen, Z.}, \bibinfo{author}{Yao, J.}, \bibinfo{author}{Xiao,
  G.}, \bibinfo{author}{Wang, S.}, \bibinfo{year}{2021}.
\newblock \bibinfo{title}{Efficient and differentiable low-rank matrix
  completion with back propagation}.
\newblock \bibinfo{journal}{IEEE Transactions on Multimedia} .
\bibitem[{Cho et~al.(2019)Cho, Choi, Park, Shin and Choo}]{cho2019image}
\bibinfo{author}{Cho, W.}, \bibinfo{author}{Choi, S.}, \bibinfo{author}{Park,
  D.K.}, \bibinfo{author}{Shin, I.}, \bibinfo{author}{Choo, J.},
  \bibinfo{year}{2019}.
\newblock \bibinfo{title}{Image-to-image translation via group-wise deep
  whitening-and-coloring transformation}, in: \bibinfo{booktitle}{Proceedings
  of the IEEE/CVF conference on computer vision and pattern recognition}, pp.
  \bibinfo{pages}{10639--10647}.
\bibitem[{Hu et~al.(2021)Hu, Nie, Wang and Li}]{hu2021low}
\bibinfo{author}{Hu, Z.}, \bibinfo{author}{Nie, F.}, \bibinfo{author}{Wang,
  R.}, \bibinfo{author}{Li, X.}, \bibinfo{year}{2021}.
\newblock \bibinfo{title}{Low rank regularization: A review}.
\newblock \bibinfo{journal}{Neural Networks} \bibinfo{volume}{136},
  \bibinfo{pages}{218--232}.
\bibitem[{Huang et~al.(2018)Huang, Yang, Lang and Deng}]{huang2018decorrelated}
\bibinfo{author}{Huang, L.}, \bibinfo{author}{Yang, D.}, \bibinfo{author}{Lang,
  B.}, \bibinfo{author}{Deng, J.}, \bibinfo{year}{2018}.
\newblock \bibinfo{title}{Decorrelated batch normalization}, in:
  \bibinfo{booktitle}{Proceedings of the IEEE Conference on Computer Vision and
  Pattern Recognition}, pp. \bibinfo{pages}{791--800}.
\bibitem[{Huang et~al.(2021)Huang, Ke, Cui, Cheng, Qiu, Jia, Ying, Zhu and
  Liang}]{huang2021deep}
\bibinfo{author}{Huang, W.}, \bibinfo{author}{Ke, Z.}, \bibinfo{author}{Cui,
  Z.X.}, \bibinfo{author}{Cheng, J.}, \bibinfo{author}{Qiu, Z.},
  \bibinfo{author}{Jia, S.}, \bibinfo{author}{Ying, L.}, \bibinfo{author}{Zhu,
  Y.}, \bibinfo{author}{Liang, D.}, \bibinfo{year}{2021}.
\newblock \bibinfo{title}{Deep low-rank plus sparse network for dynamic {MR}
  imaging}.
\newblock \bibinfo{journal}{Medical Image Analysis} \bibinfo{volume}{73},
  \bibinfo{pages}{102190}.
\bibitem[{Imran et~al.(2022)Imran, Tahir, Khalid and Uppal}]{imran2022deep}
\bibinfo{author}{Imran, S.}, \bibinfo{author}{Tahir, M.},
  \bibinfo{author}{Khalid, Z.}, \bibinfo{author}{Uppal, M.},
  \bibinfo{year}{2022}.
\newblock \bibinfo{title}{A deep unfolded prior-aided {RPCA} network for cloud
  removal}.
\newblock \bibinfo{journal}{IEEE Signal Processing Letters}
  \bibinfo{volume}{29}, \bibinfo{pages}{2048--2052}.
\bibitem[{Ionescu et~al.(2015)Ionescu, Vantzos and
  Sminchisescu}]{ionescu2015matrix}
\bibinfo{author}{Ionescu, C.}, \bibinfo{author}{Vantzos, O.},
  \bibinfo{author}{Sminchisescu, C.}, \bibinfo{year}{2015}.
\newblock \bibinfo{title}{Matrix backpropagation for deep networks with
  structured layers}, in: \bibinfo{booktitle}{Proceedings of the IEEE
  international conference on computer vision}, pp.
  \bibinfo{pages}{2965--2973}.
\bibitem[{Ke et~al.(2021)Ke, Huang, Cui, Cheng, Jia, Wang, Liu, Zheng, Ying,
  Zhu et~al.}]{ke2021learned}
\bibinfo{author}{Ke, Z.}, \bibinfo{author}{Huang, W.}, \bibinfo{author}{Cui,
  Z.X.}, \bibinfo{author}{Cheng, J.}, \bibinfo{author}{Jia, S.},
  \bibinfo{author}{Wang, H.}, \bibinfo{author}{Liu, X.},
  \bibinfo{author}{Zheng, H.}, \bibinfo{author}{Ying, L.},
  \bibinfo{author}{Zhu, Y.}, et~al., \bibinfo{year}{2021}.
\newblock \bibinfo{title}{Learned low-rank priors in dynamic {MR} imaging}.
\newblock \bibinfo{journal}{IEEE Transactions on Medical Imaging}
  \bibinfo{volume}{40}, \bibinfo{pages}{3698--3710}.
\bibitem[{Kessy et~al.(2018)Kessy, Lewin and Strimmer}]{kessy2018optimal}
\bibinfo{author}{Kessy, A.}, \bibinfo{author}{Lewin, A.},
  \bibinfo{author}{Strimmer, K.}, \bibinfo{year}{2018}.
\newblock \bibinfo{title}{Optimal whitening and decorrelation}.
\newblock \bibinfo{journal}{The American Statistician} \bibinfo{volume}{72},
  \bibinfo{pages}{309--314}.
\bibitem[{Krizhevsky et~al.(2009)Krizhevsky, Hinton
  et~al.}]{krizhevsky2009learning}
\bibinfo{author}{Krizhevsky, A.}, \bibinfo{author}{Hinton, G.}, et~al.,
  \bibinfo{year}{2009}.
\newblock \bibinfo{title}{Learning multiple layers of features from tiny
  images} .
\bibitem[{Liang et~al.(2019)Liang, Cheng, Ke and Ying}]{liang2019deep}
\bibinfo{author}{Liang, D.}, \bibinfo{author}{Cheng, J.}, \bibinfo{author}{Ke,
  Z.}, \bibinfo{author}{Ying, L.}, \bibinfo{year}{2019}.
\newblock \bibinfo{title}{Deep {MRI} reconstruction: unrolled optimization
  algorithms meet neural networks}.
\newblock \bibinfo{journal}{arXiv preprint arXiv:1907.11711} .
\bibitem[{Liang(2007)}]{liang2007spatiotemporal}
\bibinfo{author}{Liang, Z.P.}, \bibinfo{year}{2007}.
\newblock \bibinfo{title}{Spatiotemporal imagingwith partially separable
  functions}, in: \bibinfo{booktitle}{2007 4th IEEE international symposium on
  biomedical imaging: from nano to macro}, \bibinfo{organization}{IEEE}. pp.
  \bibinfo{pages}{988--991}.
\bibitem[{Liu et~al.(2019)Liu, Cheng, Ma, Fan and Luo}]{liu2019deep}
\bibinfo{author}{Liu, R.}, \bibinfo{author}{Cheng, S.}, \bibinfo{author}{Ma,
  L.}, \bibinfo{author}{Fan, X.}, \bibinfo{author}{Luo, Z.},
  \bibinfo{year}{2019}.
\newblock \bibinfo{title}{Deep proximal unrolling: Algorithmic framework,
  convergence analysis and applications}.
\newblock \bibinfo{journal}{IEEE Transactions on Image Processing}
  \bibinfo{volume}{28}, \bibinfo{pages}{5013--5026}.
\bibitem[{Lu et~al.(2019)Lu, Feng, Chen, Liu, Lin and Yan}]{lu2019tensor}
\bibinfo{author}{Lu, C.}, \bibinfo{author}{Feng, J.}, \bibinfo{author}{Chen,
  Y.}, \bibinfo{author}{Liu, W.}, \bibinfo{author}{Lin, Z.},
  \bibinfo{author}{Yan, S.}, \bibinfo{year}{2019}.
\newblock \bibinfo{title}{Tensor robust principal component analysis with a new
  tensor nuclear norm}.
\newblock \bibinfo{journal}{IEEE transactions on pattern analysis and machine
  intelligence} \bibinfo{volume}{42}, \bibinfo{pages}{925--938}.
\bibitem[{Mai et~al.(2022)Mai, Lam and Lee}]{mai2022deep}
\bibinfo{author}{Mai, T.T.N.}, \bibinfo{author}{Lam, E.Y.},
  \bibinfo{author}{Lee, C.}, \bibinfo{year}{2022}.
\newblock \bibinfo{title}{Deep unrolled low-rank tensor completion for high
  dynamic range imaging}.
\newblock \bibinfo{journal}{IEEE Transactions on Image Processing}
  \bibinfo{volume}{31}, \bibinfo{pages}{5774--5787}.
\bibitem[{Monga et~al.(2021)Monga, Li and Eldar}]{monga2021algorithm}
\bibinfo{author}{Monga, V.}, \bibinfo{author}{Li, Y.}, \bibinfo{author}{Eldar,
  Y.C.}, \bibinfo{year}{2021}.
\newblock \bibinfo{title}{Algorithm unrolling: Interpretable, efficient deep
  learning for signal and image processing}.
\newblock \bibinfo{journal}{IEEE Signal Processing Magazine}
  \bibinfo{volume}{38}, \bibinfo{pages}{18--44}.
\bibitem[{Nakatsukasa and Higham(2013)}]{nakatsukasa2013stable}
\bibinfo{author}{Nakatsukasa, Y.}, \bibinfo{author}{Higham, N.J.},
  \bibinfo{year}{2013}.
\newblock \bibinfo{title}{Stable and efficient spectral divide and conquer
  algorithms for the symmetric eigenvalue decomposition and the {SVD}}.
\newblock \bibinfo{journal}{SIAM Journal on Scientific Computing}
  \bibinfo{volume}{35}, \bibinfo{pages}{A1325--A1349}.
\bibitem[{Peng et~al.(2016)Peng, Li, Ling, Hu, Xiong and
  Maybank}]{peng2016salient}
\bibinfo{author}{Peng, H.}, \bibinfo{author}{Li, B.}, \bibinfo{author}{Ling,
  H.}, \bibinfo{author}{Hu, W.}, \bibinfo{author}{Xiong, W.},
  \bibinfo{author}{Maybank, S.J.}, \bibinfo{year}{2016}.
\newblock \bibinfo{title}{Salient object detection via structured matrix
  decomposition}.
\newblock \bibinfo{journal}{IEEE transactions on pattern analysis and machine
  intelligence} \bibinfo{volume}{39}, \bibinfo{pages}{818--832}.
\bibitem[{Picard(2021)}]{picard2021torch}
\bibinfo{author}{Picard, D.}, \bibinfo{year}{2021}.
\newblock \bibinfo{title}{Torch. manual\_seed (3407) is all you need: On the
  influence of random seeds in deep learning architectures for computer
  vision}.
\newblock \bibinfo{journal}{arXiv preprint arXiv:2109.08203} .
\bibitem[{Ristea et~al.(2021)Ristea, Anghel, Ionescu and
  Eldar}]{ristea2021automotive}
\bibinfo{author}{Ristea, N.C.}, \bibinfo{author}{Anghel, A.},
  \bibinfo{author}{Ionescu, R.T.}, \bibinfo{author}{Eldar, Y.C.},
  \bibinfo{year}{2021}.
\newblock \bibinfo{title}{Automotive radar interference mitigation with
  unfolded robust {PCA} based on residual overcomplete auto-encoder blocks},
  in: \bibinfo{booktitle}{Proceedings of the IEEE/CVF Conference on Computer
  Vision and Pattern Recognition}, pp. \bibinfo{pages}{3209--3214}.
\bibitem[{Shanmugam and Kalyani(2023)}]{shanmugam2023unrolling}
\bibinfo{author}{Shanmugam, S.}, \bibinfo{author}{Kalyani, S.},
  \bibinfo{year}{2023}.
\newblock \bibinfo{title}{Unrolling {SVT} to obtain computationally efficient
  {SVT} for n-qubit quantum state tomography}.
\newblock \bibinfo{journal}{IEEE Transactions on Signal Processing}
  \bibinfo{volume}{71}, \bibinfo{pages}{178--188}.
\bibitem[{Solomon et~al.(2019)Solomon, Cohen, Zhang, Yang, He, Luo, van Sloun
  and Eldar}]{solomon2019deep}
\bibinfo{author}{Solomon, O.}, \bibinfo{author}{Cohen, R.},
  \bibinfo{author}{Zhang, Y.}, \bibinfo{author}{Yang, Y.}, \bibinfo{author}{He,
  Q.}, \bibinfo{author}{Luo, J.}, \bibinfo{author}{van Sloun, R.J.},
  \bibinfo{author}{Eldar, Y.C.}, \bibinfo{year}{2019}.
\newblock \bibinfo{title}{Deep unfolded robust {PCA} with application to
  clutter suppression in ultrasound}.
\newblock \bibinfo{journal}{IEEE transactions on medical imaging}
  \bibinfo{volume}{39}, \bibinfo{pages}{1051--1063}.
\bibitem[{Sui et~al.(2018)Sui, Tang, Zhang and Wang}]{sui2018visual}
\bibinfo{author}{Sui, Y.}, \bibinfo{author}{Tang, Y.}, \bibinfo{author}{Zhang,
  L.}, \bibinfo{author}{Wang, G.}, \bibinfo{year}{2018}.
\newblock \bibinfo{title}{Visual tracking via subspace learning: A
  discriminative approach}.
\newblock \bibinfo{journal}{International Journal of Computer Vision}
  \bibinfo{volume}{126}, \bibinfo{pages}{515--536}.
\bibitem[{Townsend(2016)}]{townsend2016differentiating}
\bibinfo{author}{Townsend, J.}, \bibinfo{year}{2016}.
\newblock \bibinfo{title}{Differentiating the singular value decomposition}.
\newblock \bibinfo{type}{Technical Report}. Technical Report 2016,
  https://j-towns.github.io/papers/svd-derivative.pdf.
\bibitem[{Van~Luong et~al.(2021)Van~Luong, Joukovsky, Eldar and
  Deligiannis}]{van2021deep}
\bibinfo{author}{Van~Luong, H.}, \bibinfo{author}{Joukovsky, B.},
  \bibinfo{author}{Eldar, Y.C.}, \bibinfo{author}{Deligiannis, N.},
  \bibinfo{year}{2021}.
\newblock \bibinfo{title}{A deep-unfolded reference-based {RPCA} network for
  video foreground-background separation}, in: \bibinfo{booktitle}{2020 28th
  European Signal Processing Conference (EUSIPCO)},
  \bibinfo{organization}{IEEE}. pp. \bibinfo{pages}{1432--1436}.
\bibitem[{Wang et~al.(2019)Wang, Dang, Hu, Fua and
  Salzmann}]{wang2019backpropagation}
\bibinfo{author}{Wang, W.}, \bibinfo{author}{Dang, Z.}, \bibinfo{author}{Hu,
  Y.}, \bibinfo{author}{Fua, P.}, \bibinfo{author}{Salzmann, M.},
  \bibinfo{year}{2019}.
\newblock \bibinfo{title}{Backpropagation-friendly eigendecomposition}.
\newblock \bibinfo{journal}{Advances in Neural Information Processing Systems}
  \bibinfo{volume}{32}.
\bibitem[{Wang et~al.(2021)Wang, Dang, Hu, Fua and Salzmann}]{wang2021robust}
\bibinfo{author}{Wang, W.}, \bibinfo{author}{Dang, Z.}, \bibinfo{author}{Hu,
  Y.}, \bibinfo{author}{Fua, P.}, \bibinfo{author}{Salzmann, M.},
  \bibinfo{year}{2021}.
\newblock \bibinfo{title}{Robust differentiable {SVD}}.
\newblock \bibinfo{journal}{IEEE Transactions on Pattern Analysis and Machine
  Intelligence} \bibinfo{volume}{44}, \bibinfo{pages}{5472--5487}.
\bibitem[{Xu et~al.(2020)Xu, Zhu, Cheng, Li and Sun}]{xu2020multi}
\bibinfo{author}{Xu, Y.}, \bibinfo{author}{Zhu, L.}, \bibinfo{author}{Cheng,
  Z.}, \bibinfo{author}{Li, J.}, \bibinfo{author}{Sun, J.},
  \bibinfo{year}{2020}.
\newblock \bibinfo{title}{Multi-feature discrete collaborative filtering for
  fast cold-start recommendation}, in: \bibinfo{booktitle}{Proceedings of the
  AAAI conference on artificial intelligence}, pp. \bibinfo{pages}{270--278}.
\bibitem[{Yao et~al.(2018)Yao, Xu, Huang and Huang}]{yao2018efficient}
\bibinfo{author}{Yao, J.}, \bibinfo{author}{Xu, Z.}, \bibinfo{author}{Huang,
  X.}, \bibinfo{author}{Huang, J.}, \bibinfo{year}{2018}.
\newblock \bibinfo{title}{An efficient algorithm for dynamic {MRI} using
  low-rank and total variation regularizations}.
\newblock \bibinfo{journal}{Medical image analysis} \bibinfo{volume}{44},
  \bibinfo{pages}{14--27}.
\bibitem[{Zeng et~al.(2020)Zeng, Xie, Cui, Yin and
  Ning}]{zeng2020hyperspectral}
\bibinfo{author}{Zeng, H.}, \bibinfo{author}{Xie, X.}, \bibinfo{author}{Cui,
  H.}, \bibinfo{author}{Yin, H.}, \bibinfo{author}{Ning, J.},
  \bibinfo{year}{2020}.
\newblock \bibinfo{title}{Hyperspectral image restoration via global l 1-2
  spatial--spectral total variation regularized local low-rank tensor
  recovery}.
\newblock \bibinfo{journal}{IEEE Transactions on Geoscience and Remote Sensing}
  \bibinfo{volume}{59}, \bibinfo{pages}{3309--3325}.
\bibitem[{Zhang et~al.(2023a)Zhang, Chen, Xiong and Zhang}]{zhang2023physics}
\bibinfo{author}{Zhang, J.}, \bibinfo{author}{Chen, B.},
  \bibinfo{author}{Xiong, R.}, \bibinfo{author}{Zhang, Y.},
  \bibinfo{year}{2023}a.
\newblock \bibinfo{title}{Physics-inspired compressive sensing: Beyond deep
  unrolling}.
\newblock \bibinfo{journal}{IEEE Signal Processing Magazine}
  \bibinfo{volume}{40}, \bibinfo{pages}{58--72}.
\bibitem[{Zhang et~al.(2023b)Zhang, Li and Hu}]{zhang2023dynamic}
\bibinfo{author}{Zhang, Y.}, \bibinfo{author}{Li, P.}, \bibinfo{author}{Hu,
  Y.}, \bibinfo{year}{2023}b.
\newblock \bibinfo{title}{Dynamic {MRI} using learned transform-based tensor
  low-rank network ({LT2LR-NET})}, in: \bibinfo{booktitle}{2023 IEEE 20th
  International Symposium on Biomedical Imaging (ISBI)},
  \bibinfo{organization}{IEEE}. pp. \bibinfo{pages}{1--4}.
\bibitem[{Zhang et~al.(2024)Zhang, Li and Hu}]{zhang2022t}
\bibinfo{author}{Zhang, Y.}, \bibinfo{author}{Li, P.}, \bibinfo{author}{Hu,
  Y.}, \bibinfo{year}{2024}.
\newblock \bibinfo{title}{{T2LR-Net}: An unrolling network learning transformed
  tensor low-rank prior for dynamic mr image reconstruction}.
\newblock \bibinfo{journal}{Computers in Biology and Medicine}
  \bibinfo{volume}{170}, \bibinfo{pages}{108034}.
\bibitem[{Zhang et~al.(2020)Zhang, Liu, Wu and Walid}]{zhang2020video}
\bibinfo{author}{Zhang, Y.}, \bibinfo{author}{Liu, X.Y.}, \bibinfo{author}{Wu,
  B.}, \bibinfo{author}{Walid, A.}, \bibinfo{year}{2020}.
\newblock \bibinfo{title}{Video synthesis via transform-based tensor neural
  network}, in: \bibinfo{booktitle}{Proceedings of the 28th ACM International
  Conference on Multimedia}, pp. \bibinfo{pages}{2454--2462}.
\bibitem[{Zhao et~al.(2023)Zhao, Ulfarsson and
  Sigurdsson}]{zhao2023hyperspectral}
\bibinfo{author}{Zhao, B.}, \bibinfo{author}{Ulfarsson, M.O.},
  \bibinfo{author}{Sigurdsson, J.}, \bibinfo{year}{2023}.
\newblock \bibinfo{title}{Hyperspectral image denoising using low-rank and
  sparse model based deep unrolling}, in: \bibinfo{booktitle}{IGARSS 2023-2023
  IEEE International Geoscience and Remote Sensing Symposium},
  \bibinfo{organization}{IEEE}. pp. \bibinfo{pages}{5818--5821}.

\end{thebibliography}

\end{document}